\documentclass[
reqno]{amsart}

\usepackage{latexsym,amsmath,amssymb,amscd}
\usepackage[all]{xy}
\usepackage{verbatim,color}

\def\today{\ifcase \month \or
January \or February \or March \or April \or
   May \or June \or July \or August \or
   September \or October \or November \or December \fi
   space\number\day , \number\year}

\makeatletter

  \newcommand\@dotsep{4.5}
  \def\@tocline#1#2#3#4#5#6#7{\relax
 \ifnum #1>\c@tocdepth 
 \else
\par \addpenalty\@secpenalty\addvspace{#2}%
   
   \begingroup \hyphenpenalty\@M
     \@ifempty{#4}{%
     \@tempdima\csname r@tocindent\number#1\endcsname\relax
        }{%
         \@tempdima#4\relax
           }%
      \parindent\z@ \leftskip#3\relax \advance\leftskip\@tempdima\relax
      \rightskip\@pnumwidth plus1em \parfillskip-\@pnumwidth
       #5\leavevmode\hskip-\@tempdima #6\relax
       \leaders\hbox{$\m@th
       \mkern \@dotsep mu\hbox{.}\mkern \@dotsep mu$}\hfill
       \hbox to\@pnumwidth{\@tocpagenum{#7}}\par
       \nobreak
        \endgroup
         \fi}
\makeatother

\begin{document}

\makeatletter
\@addtoreset{figure}{section}
\def\thefigure{\thesection.\@arabic\c@figure}
\def\fps@figure{h,t}
\@addtoreset{table}{bsection}

\def\thetable{\thesection.\@arabic\c@table}
\def\fps@table{h, t}
\@addtoreset{equation}{section}
\def\theequation{
\arabic{equation}}
\makeatother

\newcommand{\bfi}{\bfseries\itshape}

\newtheorem{theorem}{Theorem}
\newtheorem{corollary}[theorem]{Corollary}
\newtheorem{definition}[theorem]{Definition}
\newtheorem{example}[theorem]{Example}
\newtheorem{lemma}[theorem]{Lemma}
\newtheorem{notation}[theorem]{Notation}
\newtheorem{problem}[theorem]{Problem}
\newtheorem{proposition}[theorem]{Proposition}
\newtheorem{question}[theorem]{Question}
\newtheorem{remark}[theorem]{Remark}
\newtheorem{setting}[theorem]{Setting}
\numberwithin{theorem}{section}
\numberwithin{equation}{section}

\newcommand{\todo}[1]{\vspace{5 mm}\par \noindent
\framebox{\begin{minipage}[c]{0.85 \textwidth}
\tt #1 \end{minipage}}\vspace{5 mm}\par}

\renewcommand{\1}{{\bf 1}}

\newcommand{\hotimes}{\widehat\otimes}

\newcommand{\Ad}{{\rm Ad}}
\newcommand{\Alt}{{\rm Alt}\,}
\newcommand{\Ci}{{\mathcal C}^\infty}
\newcommand{\comp}{\circ}
\newcommand{\wt}{\widetilde}

\newcommand{\CC}{{\mathbb C}}
\newcommand{\DD}{{\mathbb D}}
\newcommand{\Hb}{\text{\bf H}}
\newcommand{\N}{\text{\bf N}}
\newcommand{\RR}{{\mathbb R}}
\newcommand{\T}{\text{\bf T}}
\newcommand{\Ub}{\text{\bf U}}

\newcommand{\ph}{\text{\bf P}}
\newcommand{\de}{{\rm d}}
\newcommand{\ev}{{\rm ev}}
\newcommand{\fimes}{\mathop{\times}\limits}
\newcommand{\id}{{\rm id}}
\newcommand{\ie}{{\rm i}}
\newcommand{\Cp}{\textbf {CPos}\,}
\newcommand{\End}{{\rm End}\,}
\newcommand{\Gr}{{\rm Gr}}
\newcommand{\GL}{{\rm GL}}
\newcommand{\Herm}{{\rm Herm}}
\newcommand{\Hilb}{{\bf Hilb}\,}
\newcommand{\Hom}{{\rm Hom}}
\renewcommand{\Im}{{\rm Im}}
\newcommand{\Ker}{{\rm Ker}\,}
\newcommand{\GLH}{\textbf{GrLHer}}
\newcommand{\HLH}{\textbf{HomogLHer}}
\newcommand{\LH}{\textbf{LHer}}
\newcommand{\Kern}{\textbf {Kern}}
\newcommand{\Lie}{\textbf{L}}
\newcommand{\lf}{{\rm l}}
\newcommand{\pr}{{\rm pr}}
\newcommand{\Ran}{{\rm Ran}\,}
\newcommand{\sa}{{\rm sa}}
\newcommand{\Skew}{{\rm Skew}}
\newcommand{\Symm}{{\rm Symm}}

\newcommand{\RK}{{\mathcal P}{\mathcal K}^{-*}}
\newcommand{\spann}{{\rm span}}
\newcommand{\SLH}{\textbf {StLHer}}
\newcommand{\Rg}{\textbf {RepGLH}}
\newcommand{\Rep}{\textbf {sRep}}

\newcommand{\Tr}{{\rm Tr}\,}
\newcommand{\Tran}{\textbf{Trans}}

\newcommand{\G}{{\rm G}}
\newcommand{\U}{{\rm U}}
\newcommand{\Gl}{{\rm GL}}
\newcommand{\SL}{{\rm SL}}
\newcommand{\SU}{{\rm SU}}
\newcommand{\VB}{{\rm VB}}

\newcommand{\Ac}{{\mathcal A}}
\newcommand{\Bc}{{\mathcal B}}
\newcommand{\Cc}{{\mathcal C}}
\newcommand{\Dc}{{\mathcal D}}
\newcommand{\Ec}{{\mathcal E}}
\newcommand{\Fc}{{\mathcal F}}
\newcommand{\Gc}{{\mathcal G}}
\newcommand{\Hc}{{\mathcal H}}
\newcommand{\Kc}{{\mathcal K}}
\newcommand{\Lc}{{\mathcal L}}
\newcommand{\Nc}{{\mathcal N}}
\newcommand{\Oc}{{\mathcal O}}
\newcommand{\Pc}{{\mathcal P}}
\newcommand{\Qc}{{\mathcal Q}}
\newcommand{\Rc}{{\mathcal R}}
\newcommand{\Sc}{{\mathcal S}}
\newcommand{\Tc}{{\mathcal T}}
\newcommand{\Uc}{{\mathcal U}}
\newcommand{\Vc}{{\mathcal V}}
\newcommand{\Xc}{{\mathcal X}}
\newcommand{\Yc}{{\mathcal Y}}
\newcommand{\Zc}{{\mathcal Z}}

\newcommand{\Ag}{{\mathfrak A}}
\newcommand{\hg}{{\mathfrak h}}
\newcommand{\mg}{{\mathfrak m}}
\newcommand{\nng}{{\mathfrak n}}
\newcommand{\pg}{{\mathfrak p}}
\newcommand{\ug}{{\mathfrak u}}
\newcommand{\Sg}{{\mathfrak S}}

\newcommand{\blue}[1]{\textcolor{blue}{#1}}
\newcommand{\red}[1]{\textcolor{red}{#1}}
\newcommand{\green}[1]{\textcolor{green}{#1}}

\markboth{}{}

\makeatletter
\title[Reproducing kernels and positivity of vector bundles]
{Reproducing kernels and positivity \\ of vector bundles in infinite dimensions}
\author{Daniel Belti\c t\u a and  Jos\'e E. Gal\'e}
\address{Institute of Mathematics ``Simion
Stoilow'' of the Romanian Academy,
P.O. Box 1-764, Bucharest, Romania}
\email{beltita@gmail.com, Daniel.Beltita@imar.ro}
\address{Departamento de matem\'aticas and I.U.M.A.,
Universidad de Zaragoza, 50009 Zaragoza, Spain}
\email{gale@unizar.es}
\thanks{This research was partly  supported by Project MTM2010-16679, DGI-FEDER, of the MCYT, Spain. 
Partial financial support is also acknowledged by the first-named author from the Grant
of the Romanian National Authority for Scientific Research, CNCS-UEFISCDI,
project number PN-II-ID-PCE-2011-3-0131, 
and by the second-named author from Project E-64, D.G. Arag\'on, Spain.}
\date{February 3, 2014}

\makeatother

\begin{abstract}

We investigate the interaction between the existence of reproducing kernels 
on infinite-dimensional Hermitian vector bundles and 
the positivity properties of the corresponding bundles. 
The positivity refers to the curvature form of certain covariant derivatives 
associated to reproducing kernels on the vector bundles under consideration. 
The values of the curvature form are Hilbert space operators, 
and its positivity is thus understood in the usual sense from operator theory. 
\end{abstract}

\maketitle
\tableofcontents

\section{Introduction}
\label{intro}
The idea of positivity plays a central role in Hilbert space operator theory, 
see for instance the whole panel of constructions of Gelfand-Naimark-Segal type, 
in particular dilation theory of completely positive maps or the theory of reproducing kernel Hilbert spaces. 

On the other hand, ideas of positivity and order structures also hold a quite important place in 
branches of mathematics which might seem to be remote from operator theory, 
as it is the case with the complex algebraic geometry; 
see for instance the impressive two-volume treatise \cite{La04}, 
which was devoted to a thorough discussion of that topic. 
We will show in this paper that these ideas of positivity are actually rather close to each other, 
which sheds a fresh light on the relationship between 
the aforementioned theories 
and raises several interesting problems in both these research areas. 

Maybe we should also mention at this point that 
the interaction between complex geometry and operator theory 
has already surfaced in the literature, 
as for instance in the famous theory of the Cowen-Douglas operators (\cite{CD78}), 
without to emphasize particularly the idea of positivity. 
As regards the specific topic of the present paper, we will discuss the relationship between 
the Griffiths positivity of holomorphic vector bundles 
(\cite{Gr69}, \cite{GH78}) 
and the reproducing kernels on infinite-dimensional vector bundles that we have studied recently 
(\cite{BG08}, \cite{BG09}, \cite{BG11}, \cite{BG13}). 

Reproducing kernels and covariant derivatives often occur simultaneously 
on the vector bundles involved 
in various problems in areas such as the geometric quantization, the geometric representation theory of Lie groups, the theory of Cowen-Douglas operators  etc. 
This simple remark additionally motivated the present note, 
as we tried to provide an explanation for the aforementioned occurrence by using the universality properties 
of the reproducing kernels established in our previous paper \cite{BG11}. 
A related issue is the infinite-dimensional extension of 
the Chern correspondence between the connections and the (almost) complex structures 
on the total space of a bundle, 
that we plan to consider in a forthcoming work; 
see Problem~\ref{forth} below for some more details and references.

The simplest setting that illustrates that idea is provided by 
the \emph{tautological vector bundle} $\Pi_\Hc$ corresponding to any complex Hilbert space~$\Hc$. 
Let $\Gr(\Hc)$ be the Grassmann manifold, 
whose points are the closed linear subspaces of~$\Hc$, 
and define 
$\Tc(\Hc)=\{(\Sc,x)\in\Gr(\Hc)\times\Hc\mid x\in\Sc\}$
and $\Pi_\Hc\colon\Tc(\Hc)\to\Gr(\Hc)$, 
$\Pi_\Hc(\Sc,x)=\Sc$. 
Both $\Gr(\Hc)$ and $\Tc(\Hc)$ are complex Banach manifolds and the mapping $\Pi_\Hc$ is a Hermitian holomorphic vector bundle. 
For the purposes of the present investigation 
 it is important to single out two additional structures on this bundle: 

(1) Let $p_1,p_2\colon\Gr(\Hc)\times\Gr(\Hc)\to\Gr(\Hc)$ 
the natural Cartesian projections. 
For any $\Sc\in\Gr(\Hc)$ let us also denote by $p_\Sc$ the orthogonal projection of $\Hc$ onto~$\Sc$. 
Then we can define 
$$
\begin{aligned}
Q_\Hc & \colon\Gr(\Hc)\times\Gr(\Hc) \to{\rm Hom}(p_2^*(\Pi_\Hc),p_1^*(\Pi_\Hc)),\\ 
Q_\Hc&(\Sc_1,\Sc_2)=(p_{\Sc_1})|_{\Sc_2}\colon  \Sc_2\to\Sc_1.
\end{aligned}$$
The mapping $Q_{\Hc}$ is the {\it universal reproducing kernel} 
associated with the Hilbert space $\Hc$ 
(see \cite{BG11} and Example~\ref{ex3} below). 

(2) Just as in the finite-dimensional case, 
there exists a unique {\it Chern connection}~on $\Pi_\Hc$, 
that is, a linear connection which is compatible in the usual sense 
both with the holomorphic structure and with the Hermitian structure
defined by the inner product on every fiber $\Sc$ 
(see Theorem~\ref{main_lin} below). 

Our present paper belongs to a project devoted to understanding 
the relationship between the above universal reproducing kernel 
and linear connections  and how their interaction propagates from the tautological bundle to 
pretty wide classes of reproducing kernels on Hermitian vector bundles, 
in a sense that will be made precise below. 
We already pointed out in \cite{BG13} a certain functorial correspondence 
between reproducing kernels and linear connections on infinite-dimensional vector bundles, 
and in the present paper we take up the study of that circle of ideas 
from the perspective of the complex algebraic geometry, which eventually amounts to curvature positivity properties 
of the vector bundles involved in this discussion. 

We will briefly survey some facts from our earlier papers \cite{BG08}, \cite{BG09}, \cite{BG11}, and \cite{BG13}, 
and will also raise some problems and announce a few partial results with sketchy proofs. 
Full details of these proofs will be published elsewhere. 

The paper is structured as follows.
In Section~\ref{Sect2} we provide some background information on reproducing kernels on vector bundles along with some basic examples. 
Section~\ref{Sect3} is devoted to a discussion of Chern covariant derivatives on Hermitian holomorphic vector bundles,  
which are covariant derivatives that are compatible both with the Hermitian and with the holomorphic structures of these bundles. 
In Section~\ref{Sect4} we introduce the Griffith positivity condition for bundle-valued differential 2-forms 
in our setting of infinite-dimensional bundles. 
In Section~\ref{Sect5}  we finally relate the reproducing kernels to the Griffith positivity property of Hermitian vector bundles. 
The paper concludes by Appendix~\ref{SectA} which includes some basic observations on vector-valued differential forms, 
which are needed in the main body of our paper and were collected there for the reader's convenience.

\section{Reproducing kernels on Hermitian vector bundles}\label{Sect2}

\begin{definition}\label{like}
\normalfont
Let $Z$ be a Banach manifold.
A {\it Hermitian structure} on a smooth Banach vector bundle $\Pi\colon D\to Z$ is a family
$\{(\cdot\mid\cdot)_z\}_{z\in Z}$
with the following properties:

\begin{itemize}
\item[{\rm(a)}]
For every $z\in Z$,
$(\cdot\mid\cdot)_z\colon D_z\times D_z\to{\mathbb C}$
is a scalar product (${\mathbb C}$-linear in the first variable) 
that turns the fiber $D_z$ into a complex Hilbert space. 
\item[{\rm(b)}]

If $V$ is any open subset of $Z$,
and 
$\Psi_V\colon V\times\Ec\to\Pi^{-1}(V)$
is a trivialization (whose typical fiber is the complex Hilbert space $\Ec$) 
of the vector bundle $\Pi$
over $V$,  
then the function
$(z,x,y)\mapsto(\Psi_V(z,x)\mid\Psi_V(z,y))_z$,
$V\times \Ec\times\Ec\to{\mathbb C}$
is smooth.
\end{itemize}

A {\it Hermitian bundle} is a bundle endowed with a Hermitian structure as above.
\end{definition}

\begin{definition}\label{reprokernel}
\normalfont
Let $\Pi\colon D\to Z$ be a Hermitian bundle, 
and $p_1,p_2\colon Z\times Z\to Z$ be the Cartesian projections. 
A {\it reproducing kernel} on~$\Pi$ is a continuous section of the bundle 
$\rm{Hom}(p_2^*\Pi,p_1^*\Pi)\to Z\times Z$ such that the
mappings 
$K(s,t)\colon D_t\to D_s$ ($s,t\in Z$)
are bounded linear operators and such that $K$ is positive definite in the following sense: 
For every $n\ge1$ and $t_j\in Z$,
$\eta_j\in D_{t_j}$ ($j= 1,\dots, n$),
\begin{equation}\label{reprokernel_eq1}
\sum_{j,l=1}^n
\bigl(K(t_l,t_j)\eta_j\mid \eta_l\bigr)_{t_l}\ge0.
\end{equation}

For every $\xi\in D$ we set 
$$
K_\xi:=K(\cdot,\Pi(\xi))\xi\colon Z\to D
$$ 
which is a section of the bundle $\Pi$.  
For $\xi,\eta\in D$, 
the prescriptions 
\begin{equation}\label{reprokernel_eq2}
(K_\xi\mid K_\eta)_{\Hc^K}:=(K(\Pi(\eta),\Pi(\xi))\xi\mid\eta)_{\Pi(\eta)},
\end{equation}
define an inner product 
$(\cdot\mid \cdot)_{\Hc^K}$ on $\hbox{span}\{K_\xi:\xi\in D\}$ 
whose completion gives rise to a Hilbert space denoted by $\Hc^K$, 
which consists of sections of the bundle $\Pi$. 
We also define the mappings 
$$
\begin{aligned}
\widehat{K}& \colon D\to\Hc^K,\quad \widehat{K}(\xi)=K_\xi,\\
\zeta_K & \colon Z\to\Gr(\Hc^K),\quad \zeta_K(s)=\overline{\widehat{K}(D_s)},
\end{aligned}
$$
where $\Gr(\Hc^K)$ is the Grassmann manifold of all closed subspaces of $\Hc^K$ and the bar over 
$\widehat{K}(D_s)$ indicates the topological closure. 
See \cite{BG08} for details. 
\end{definition}

\begin{example}[trivial bundles]\label{ex1}
\normalfont
Let $Z$ be any Banach manifold 
(for instance any open subset of some real Banach space) 
and $\Ec$ be any complex Hilbert space and define the trivial bundle 
$$
\Pi\colon D=Z\times\Ec\to Z,\quad (z,x)\mapsto z.
$$
For every $z\in Z$ we have $D_z=\{z\}\times\Ec$ and
moreover there exists a one-to-one correspondence $\sigma\mapsto f_\sigma$ 
sections $\sigma$ of $\Pi$ and $\Ec$-valued functions $F_\sigma$ on $Z$ 
given by 
$$
(\forall z\in Z)\quad \sigma(z)=(z,F_\sigma(z)).
$$
Denote by $\GL^{+}(\Ec)$ the set of positive invertible operators on $\Ec$, 
which is an open subset of the $C^*$-algebra $\Bc(\Ec)$. 
Then there exists a one-to-one correspondence between the Hermitian structures on $\Pi$ and 
the smooth mappings $h\colon Z\to\GL^{+}(\Ec)$ given by 
$$
(\forall z\in Z)\quad (\cdot\mid\cdot)_z\colon D_z\times D_z\to\CC,
\quad ( (z,x_1)\mid (z,x_2) )_z:=( h(z) x_1\mid x_2)_{\Ec}. 
$$
Also, there exists a one-to-one correspondence between the reproducing kernels $K$ on $\Pi$ and the 
$\Bc(\Ec)$-valued reproducing kernels $\kappa$ on $Z$ (see \cite{Ne00}) by 
$$
(\forall z_1,z_2\in Z)(\forall x\in\Ec)\quad 
K(z_1,z_2)\colon D_{z_2}\to D_{z_1},\quad (z_2,x)\mapsto(z_1,\kappa(z_1,z_2)x).
$$
\end{example}

\begin{example}[homogeneous bundles]\label{ex2}
\normalfont
Let $G_A$ be a Banach-Lie group with a Banach-Lie subgroup $G_B$.
Let 
$\rho_A\colon G_A\to \Bc(\Hc_A)$ and $\rho_B\colon G_B\to \Bc(\Hc_B)$ 
be uniformly continuous unitary representations 
with 
$\Hc_B\subseteq\Hc_A$, $\rho_B(u)=\rho_A(u)|_{\Hc_B}$ for
$u\in G_B$ and
$\Hc_A=\overline{\spann}\rho_A(G_A)\Hc_B$. 

Let us consider the homogeneous vector bundle 
$\Pi_\rho\colon G_A\times_{G_B}\Hc_B\to G_A/G_B$, induced by the representation $\rho_B$.  
Recall that 
$\G_A\times_{G_B}\Hc_B$ is the Cartesian product $\G_A\times\Hc_B$ modulo the equivalence relation 
defined by
$$
(u,h)\sim(u',h')\iff (\exists w\in G_B) \quad u'=uw,\ h'=\rho(w^{-1})h,
$$ 
endowed with its canonical structure of Banach manifold; see \cite{KM97}. 

We provide $\Pi_\rho$ with 
the 
Hermitian structure given by
$$ 
\left([(u,f)],[(u,h)]\right)_s:=(f\mid h)_{\Hc}, \quad u\in G_A,s:=u G_B, f,h\in\Hc_B.
$$
Let $P\colon\Hc_A\to\Hc_B$ be the orthogonal projection. 

We define
the reproducing kernel $K_\rho$ on the homogeneous Hermitian vector bundle
$\Pi_\rho\colon D=G_A\times_{G_B}\Hc_B\to G_A/G_B$ by
\begin{equation}\label{homoker}
K_\rho(uG_B,vG_B)[(v,f)]=[(u,P(\rho_A(u^{-1})\rho_A(v)f))],
\end{equation}
for $uG_B,vG_B\in D$ and $f\in\Hc_B$
(see \cite {BG08}). 

There exists
a unitary operator $W\colon\Hc^{K_\rho}\to\Hc_A$ such that
$W(K_\eta)=\pi_A(v)f$ whenever $\eta=[(v,f)]\in D$; 
see the end of the proof of \cite[Proposition 4.1]{BG08}.
\end{example}

\begin{example}[tautological bundles]\label{ex3}
\normalfont 
In Example~\ref{ex2} assume $G_A=U(\Hc_A)$ with the tautological representation $\rho_A$, 
and 
$$
G_B=\{u\in U(\Hc_A)\mid u(\Hc_B)\subseteq\Hc_B\}\simeq U(\Hc_B)\times U(\Hc_B^\perp).
$$ 
Denote $\Gr_{\Hc_B}(\Hc_A):=\{u(\Hc_B)\mid u\in U(\Hc_A)\}$, 
and 
$$
\Tc_{\Hc_B}(\Hc_A)=\{(u(\Hc_B),x)\mid u\in U(\Hc_A), x\in u\Hc_B\}\subseteq\Gr_{\Hc_B}(\Hc_A)\times\Hc_A.
$$ 
Then the pair of maps 
$$\begin{aligned}
G_A\times_{G_B}\Hc_B\to \Tc_{\Hc_B}(\Hc_A),&\quad [(u,x)]\mapsto(u(\Hc_B),u(x)), \\
G_A/G_B\to\Gr_{\Hc_B}(\Hc_A),&\quad uG_B\mapsto u(\Hc_B)
\end{aligned}$$
defines an isomorphism of the vector bundle $\Pi_\rho\colon G_A\times_{G_B}\Hc_B\to G_A/G_B$ 
onto the \emph{tautological bundle} 
$\Pi_{\Hc_A,\Hc_B}\colon\Tc_{\Hc_B}(\Hc_A)\to\Gr_{\Hc_B}(\Hc_A)$, $(\Sc,x)\mapsto \Sc$.
See \cite{BG09} for some more details. 

Now let $p_1,p_2\colon\Gr_{\Hc_B}(\Hc_A)\times\Gr_{\Hc_B}(\Hc_A)\to\Gr_{\Hc_B}(\Hc_A)$ 
the natural Cartesian projections. 

For any $\Sc\in\Gr_{\Hc_B}(\Hc_A)$ let us also denote by $p_\Sc\colon\Hc_A\to\Sc$ the corresponding orthogonal projection, 
whose adjoint operator is the inclusion map $p_\Sc^*=\iota_\Sc\colon\Sc\hookrightarrow\Hc_A$. 
Then we can define 
$$
\begin{aligned}
Q_{\Hc_A,\Hc_B} & \colon\Gr_{\Hc_B}(\Hc_A)\times\Gr_{\Hc_B}(\Hc_A)\to
 {\rm Hom}(p_2^*(\Pi_{\Hc_A,\Hc_B}),p_1^*(\Pi_{\Hc_A,\Hc_B})),\\
Q_{\Hc_A,\Hc_B}&(\Sc_1,\Sc_2)=p_{\Sc_1}p_{\Sc_2}^*=(p_{\Sc_1})|_{\Sc_2}\colon\Sc_2\to\Sc_1. 
\end{aligned}
$$
The mapping $Q_{\Hc_A,\Hc_B}$ is the {\it universal reproducing kernel} 
corresponding to the Hilbert space $\Hc_A$ and its closed subspace~$\Hc_B$; 
see \cite{BG11}. 
Note that $Q_{\Hc_A,\Hc_B}$ actually depends on $\Gr_{\Hc_B}(\Hc_A)$ and not on~$\Hc_B$. 
\end{example}

\section{Covariant derivatives on Hermitian vector bundles}\label{Sect3}

This section contains the main tools that allow us to describe positivity properties 
of vector bundles. 
These tools are the covariant derivatives and their curvatures. 
As we will see, in the case of holomorphic vector bundles, 
the existence of nontrivial global holomorphic cross-sections entails a positive curvature property. 
The conclusion of all that will be that 
a certain \emph{intrinsic positivity property} is necessary in order that a holomorphic vector bundle 
admits nontrivial reproducing kernels that give rise to Hilbert spaces of holomorphic cross-sections. 

\subsection*{Covariant derivative and curvature}
\begin{definition}\label{def3}
\normalfont
We first define the covariant derivatives on trivial vector bundles. 
So assume $X$ is an open subset of any real Banach space $\Xc$, and let $\Vc$ be another real Banach space. 
A \emph{linear connection form} on the trivial bundle $X\times\Vc\to X$, $(x,v)\mapsto x$, 
is any 1-form $A\in\Omega^1(X,\Bc(\Vc))$. 
The value of $A$ at any point $x\in X$ is denoted by 
$A_x\in\Bc(\Xc,\Bc(\Vc))$ and we are going to use freely the natural topological isomorphisms
$$
\Bc(\Xc,\Bc(\Vc))\simeq\Bc(\Xc,\Vc;\Vc)\simeq \Bc(\Xc\hotimes\Vc,\Vc)
$$
where $\Bc(\Xc,\Vc;\Vc)$ stands for the space of bounded bilinear maps from $\Xc\times\Vc$ into~$\Vc$ 
and $\hotimes$ denotes the projective tensor product of Banach spaces. 

The \emph{covariant derivative} corresponding to the above linear connection form 
is the sequence of linear operators $\nabla\colon\Omega^p(X,\Vc)\to\Omega^{p+1}(X,\Vc)$ 
defined for $p=0,1,2,\dots$ by 
$$\nabla\sigma=\de\sigma+A\wedge\sigma $$
for every $\sigma\in\Omega^p(X,\Vc)$, where the wedge product 
$$\wedge \colon \Omega^1(X,\Bc(\Vc))\times\Omega^p(X,\Vc)\to\Omega^{p+1}(X,\Vc)$$
is defined (see Definition~\ref{def2}) by using the natural bilinear map $\Bc(\Vc)\times\Vc\to\Vc$ given 
by the action of the operators in $\Bc(\Vc)$ on $\Vc$.

If $\Pi\colon D\to Z$ is any (locally trivial) vector bundle, 
then for every $p=0,1,2,\dots$ we define 
$\Hom(\wedge^p\tau_Z,\Pi)$ as the vector bundle over $Z$ whose fiber 
over $z\in Z$ is the space $\Bc(\wedge T_z Z, D_z)$ 
of all bounded skew-symmetric $p$-linear maps $T_zZ\times\cdots\times T_zZ\to D_z$  
(see the Appendix below).  
We denote by $\Omega^p(Z,D)$ the space of all locally defined smooth sections of $\Hom(\wedge^p\tau_Z,\Pi)$.

A \emph{covariant derivative} on the vector bundle $\Pi$  
is any sequence of operators $\nabla\colon\Omega^p(Z,D)\to\Omega^{p+1}(Z,D)$ for 
$p=0,1,2,\dots$ which can be expressed in terms of connection forms as above in any local trivialization of~$\Pi$. 
\end{definition}

\begin{remark}
\normalfont
In Definition~\ref{def3} we have 
$$\begin{aligned}
(\nabla\sigma)_x(x_1,\dots,x_{p+1})
=&(\de_x\sigma)(x_1,\dots,x_{p+1}) \\
&+\sum_{j=1}^{p+1}\underbrace{A_x(x_j)}_{\hskip20pt\in\Bc(\Vc)}
\underbrace{\sigma_x(x_1,\dots,x_{j-1},x_{j+1},\dots,x_{p+1})}_{\hskip15pt\in\Vc}
\end{aligned} $$
for every $\sigma\in\Omega^p(X,\Vc)$, $x\in X$, and $x_1,\dots,x_{p+1}\in\Xc$. 
In particular, if $p=0$, then $\sigma\in\Ci(X,\Vc)$ and 
the 1-form $\nabla_\sigma\in\Omega^1(X,\Vc)$ is given by 
$$
(\nabla\sigma)_x(x_1)=(\de_x\sigma)(x_1)+A_x(x_1)\sigma_x $$
for every $x\in X$ and $x_1\in\Xc$. 
\end{remark}

\begin{definition}\label{def3.5}
\normalfont
Assume the setting of Definition~\ref{def3}. 
The \emph{curvature form} corresponding to the linear connection form $A\in\Omega^1(X,\Bc(\Vc))$ is 
$$\Theta:=\de A+A\wedge A\in\Omega^2(X,\Bc(\Vc))$$
where the wedge product 
$$\wedge \colon \Omega^1(X,\Bc(\Vc))\times\Omega^1(X,\Bc(\Vc))\to\Omega^2(X,\Bc(\Vc))$$
is defined (see Definition~\ref{def2}) via the natural bilinear map $\Bc(\Vc)\times\Bc(\Vc)\to\Bc(\Vc)$ given by the product of operators in $\Bc(\Vc)$.  

If $\Pi\colon D\to Z$ is an arbitrary (locally trivial) vector bundle 
with covariant derivative $\nabla\colon\Omega^p(Z,D)\to\Omega^{p+1}(Z,D)$
defined for $p=0,1,2,\dots$, 
then the curvature forms defined as above in local trivializations 
can be glued together into a global curvature form $\Theta\in\Omega^2(Z,\End(\Pi))$. 
The covariant derivative is said to be \emph{flat} if its curvature is $\Theta=0$. 
\end{definition}

\begin{remark}
\normalfont
In the notation of Definition~\ref{def3.5}, 
for every $z\in Z$ we have the skew-symmetric bilinear map $\Theta_z\colon T_zZ\times T_zZ\to\Bc(D_z)$. 
Moreover, for every $p\ge 0$ and $\sigma\in\Omega^p(Z,D)$ 
we have $\nabla(\nabla\sigma)=\Theta\wedge\sigma$. 
Hence the covariant derivative $\nabla$ is flat if and only if $\nabla^2=0$. 
\end{remark}

\subsection*{Covariant derivatives compatible with Hermitian structures}
\begin{definition}\label{pb2.81}
\normalfont
Let $\Pi\colon D\to Z$ be any Hermitian vector bundle. 
We say that a covariant derivative  
$\nabla$  on $\Pi$ 
is \emph{compatible with the Hermitian structure} if it satisfies the following condition: 
If $W$ is an open subset of $Z$ and $\sigma_1,\sigma_2\colon W\to D$ 
are smooth cross-sections of the bundle $\Pi$, 
then 
$$
\de(\sigma_1\mid\sigma_2)=(\nabla\sigma_1\mid\sigma_2)
+(\sigma_1\mid \nabla\sigma_2)\in\Omega^1(W,\CC) 
$$
that is, we have 
$(\de_z(\sigma_1\mid\sigma_2))(x)
=((\nabla\sigma_1)(x)\mid \sigma_2(z))_z
+(\sigma_1(z)\mid(\nabla\sigma_2)(x))_z$ 
for all $z\in W$ and $x\in T_zZ$.
This condition has a local character, so it suffices to check it in local trivializations of the bundle~$\Pi$. 
\end{definition}

\begin{remark}\label{def5}
\normalfont
Assume that $X$ is an open set in the real Banach space $\Xc$ 
and $\Ec$ is any complex Hilbert space. 
A Hermitian structure on $\Ec$ is then the same thing as a smooth mapping 
$h\colon X\to\GL^{+}(\Ec)$. 
If we define the mapping 
$$
(\cdot\mid\cdot)_x\colon\Ec\times\Ec\to \CC,\quad 
(v_1\mid v_2)_x=(h(x)v_1\mid v_2) 
$$
for every $x\in X$, then we can note the following: 
\begin{enumerate}
\item For every $x\in X$ the mapping $(\cdot\mid\cdot)_x$ is a scalar product 
compatible with the topology of $\Ec$ since $h_x\in\GL^{+}(\Ec)$.  
\item Since the Hermitian structure $h$ is smooth, we can define a natural sesquilinear map 
$$
h(\cdot,\cdot)\colon\Ci(X,\Ec)\times\Ci(X,\Ec)\to\Ci(X,\CC) 
$$
such that $(h(\phi,\psi))_x=(\phi_x\mid \psi_x)_x$ 
for all $\phi,\psi\in\Ci(X,\Ec)$ and $x\in X$. 
\item We can also define a sesquilinear map denoted in the same way, 
$$
h(\cdot,\cdot)\colon\Omega^1(X,\Ec)\times\Ci(X,\Ec)\to\Omega^1(X,\CC) 
$$
such that for all $\sigma\in\Omega^1(X,\Ec)$, $\psi\in\Ci(X,\Ec)$, and $x\in X$ we have 
$$
(h(\sigma,\psi))_x=(\sigma_x(\cdot)\mid \psi_x)_x
=(h_x\sigma_x(\cdot)\mid \psi_x)\colon\Xc\to\CC,
$$ 
where we recall that $\sigma_x\in\Bc(\Xc,\Ec)$. 
Similarly, one defines a sesquilinear map  
$$
h(\cdot,\cdot)\colon\Ci(X,\Ec)\times\Omega^1(X,\Ec)\to\Omega^1(X,\CC) $$
such that 
$$
\begin{aligned}
(h(\psi,\sigma))_x(v)
&=\overline{h(\sigma,\psi)_x(y)}
=\overline{(\sigma_x(y)\mid\psi_x)_x}
=\overline{(h_x\sigma_x(y)\mid\psi_x)} \\
&=(\psi_x\mid h_x\sigma_x(y))
\end{aligned}
$$ 
for all $\sigma\in\Omega^1(X,\Ec)$, $\psi\in\Ci(X,\Ec)$, 
$y\in\Xc$, and $x\in X$. 
\end{enumerate}
\end{remark}

The following result is suggested by the classical situation of finite-dimensional bundles; 
see for instance the computations prior to 
\cite[Ch. III, Prop. 1.11]{We08}. 

\begin{proposition}\label{hermitian}
In the setting of Remark~\ref{def5}, assume that a linear connection form $A\in\Omega^1(X,\Bc(\Ec))$ is also given. 
Then the following assertions are equivalent: 
\begin{enumerate}
\item\label{hermitian_item1} 
We have 
$$
\de(h(\phi,\psi))=h(\nabla\phi,\psi)+h(\phi,\nabla\psi) 
$$
for all $\phi,\psi\in\Ci(X,\Ec)$. 
\item\label{hermitian_item2} 
The equation 
$$
\de_x h=h_xA_x(\cdot)+A_x(\cdot)^*h_x\in\Bc(\Xc,\Bc(\Ec))
$$
is satisfied for every $x\in X$. 
\end{enumerate}
\end{proposition}

\begin{proof} 
Since the mapping $\Bc(\Ec)\times\Ec\times\Ec\to \CC$, 
$(T,v,w)\mapsto( Tv\mid w)$ is trilinear and continuous, 
it follows by the product rule of differentiation (see for instance \cite[Th. 1, Ch. 1]{Nl69}) 
that for all $\phi,\psi\in\Ci(X,\Ec)$ and $x\in X$ we have the following equalities in $\Bc(\Xc,\CC)$: 
$$
\de_x(h(\phi,\psi))
=\de_x (h\phi\mid \psi)
=(\de_x h(\cdot)\phi_x\mid \psi_x)
+( h_x\de_x\phi(\cdot)\mid \psi_x)
+( h_x\phi_x\mid \de_x\psi(\cdot)),
$$
hence by using the fact that $h_x^*=h_x$ in $\Bc(\Ec)$ we get 
$$
\de_x( h(\phi,\psi))=(\de_x h(\cdot)\phi_x\mid \psi_x)
+(h(\de\phi,\psi))_x+(h(\phi,\de\psi))_x.
$$
Since $\nabla=\de+A$, we get further 
$$
\begin{aligned}
\de_x &((h(\phi,\psi)))-(h(\nabla\phi,\psi))_x-(h(\phi,\nabla\psi))_x \\
& =(\de_x h(\cdot)\phi_x\mid \psi_x)
-(h(A\wedge\phi,\psi))_x+(h(\phi,A\wedge\psi))_x  \\
& =(\de_x h(\cdot)\phi_x\mid \psi_x)
-( h_xA_x(\cdot)\phi_x\mid \psi_x)
-( h_x\phi_x\mid A_x(\cdot)\psi_x)\\
& =(\de_x h(\cdot)\phi_x
- h_x A_x(\cdot)\phi_x
- A_x(\cdot)^*h_x\phi_x\mid \psi_x).
\end{aligned} 
$$
With this equality at hand, it follows at once that the assertions in the statement are equivalent to each other. 
\end{proof}

\subsection*{Linear connections compatible with complex structures}
\begin{definition}\label{def4}
\normalfont
Assume $X$ is any open subset of some complex Banach space $\Xc$ and $\Ec$ is another complex Banach space. 
Let $A\in\Omega^1(X,\Bc(\Ec))$ be any connection form, 
hence $A\colon TX=X\times\Xc\to\Bc(\Ec)$ is smooth and $\RR$-linear in the second variable. 
Since both $\Xc$ and $\Ec$ are complex vector spaces, 
we can use the direct sum decomposition~\eqref{dbar_eq1} 
to define the linear operators 
$$
\nabla'\colon \Ci(X,\Ec)\to\Omega^{(1,0)}(X,\Ec)
\text{ and }
\nabla''\colon \Ci(X,\Ec)\to\Omega^{(0,1)}(X,\Ec)
$$
such that 
$$
\nabla=\nabla'+\nabla''
$$
where $\nabla\colon\Ci(X,\Ec)\to\Omega^1(X,\Ec)$ is the covariant derivative corresponding to~$A$. 
So for every $\sigma\in\Ci(X,\Ec)$ and $x\in X$ we have  
$(\nabla\sigma)(x)=(\nabla'\sigma)(x)+(\nabla''\sigma)(x)$, 
the unique decomposition 
for which the operator $(\nabla'\sigma)(x)\colon\Xc\to\Bc(\Ec)$ is $\CC$-linear while  
$(\nabla''\sigma)(x)\colon\Xc\to\Bc(\Ec)$ is conjugate linear. 
\end{definition}
The following result is suggested by the beginning remark in the proof of \cite[Ch. III, Th. 2.1]{We08}. 

\begin{proposition}\label{complex}
In the setting of Definition~\ref{def4}, the following assertions are equivalent:
\begin{enumerate}
 \item\label{complex_item1} 
For every $\sigma\in\Oc(X,\Ec)$ we have $\nabla''\sigma=0$. 
 \item\label{complex_item2}  
We have $A\in\Omega^{(1,0)}(X,\Bc(\Ec))$. 
\end{enumerate}
\end{proposition}

\begin{proof}
First note that the 1-form $A$ takes values in the complex vector space $\Bc(\Ec)$, hence 
we get a decomposition $A=A^{(1,0)}+A^{(0,1)}$ (see the Appendix section), 
and then condition~\eqref{complex_item2} is equivalent to $A^{(0,1)}=0$. 

If $\sigma\in\Oc(X,\Ec)$, then 
$$
\nabla\sigma=\de\sigma+A\wedge\sigma=\partial\sigma+\bar\partial\sigma+A\wedge \sigma
=\partial\sigma+A\wedge \sigma
$$
hence, by Remark~\ref{compl}\eqref{compl_item1},
$$
\nabla'\sigma=\partial\sigma+A^{(1,0)}\wedge\sigma
$$
and 
$$
\nabla''\sigma=A^{(0,1)}\wedge\sigma.
$$
By considering constant $\Ec$-valued functions on $X$, 
we see that $\Ec$ is generated by the values of functions in $\Oc(X,\Ec)$. 
Then the above equality implies that Assertion~\eqref{complex_item1} is equivalent to $A^{(0,1)}=0$, 
and this concludes the proof. 
\end{proof}

\begin{definition}\label{compl_compat}
\normalfont
If the assertions in Proposition~\ref{complex} are satisfied, 
then we say that the linear connection corresponding to $A$ is 
compatible with the complex structures of $X$ and~$\Ec$. 
More generally, a linear connection on a holomorphic Banach vector bundle is 
\emph{compatible with the complex structure} if its local connection form in any local holomorphic trivialization is compatible (in the above sense) with the complex structures of the base and the fiber. 
It is easily seen that this property has a local character and does not depend on the choice of a local holomorphic trivialization. 
\end{definition}

\subsection*{Chern covariant derivatives}
\begin{definition}
\normalfont
A \emph{Hermitian holomorphic vector bundle} is any holomorphic vector bundle $\Pi\colon D\to Z$ 
endowed with a Hermitian structure. 
In this framework, a \emph{Chern covariant derivative} on $\Pi$ 
is any covariant derivative which is compatible  
both with the complex structure and with the Hermitian structure of the vector bundle~$\Pi$. 
\end{definition}

We are now able to prove an infinite-dimensional version of \cite[Ch. III, Th. 2.1]{We08} for trivial bundles. 

\begin{lemma}\label{chern}
Let $X$ be any open subset of some complex Banach space $\Xc$, $\Ec$ be any complex Hilbert space, and $h\colon X\to\GL^{+}(\Ec)$ be any smooth mapping. 
Then there exists a unique connection form $A\in\Omega^1(X,\Bc(\Ec))$ 
that is compatible both with the Hermitian structure given by $h$ 
and with the complex structures of $X$ and $\Ec$, and it is given by 
\begin{equation}\label{chern_eq1}
A_x=h_x^{-1}(\partial h)_x 
\end{equation}
for every $x\in X$. 
\end{lemma}

\begin{proof}
We first prove the uniqueness assertion. 
If $A$ is a connection form that satisfies the compatibility conditions mentioned in the statement, then by Proposition~\ref{hermitian}\eqref{hermitian_item2} we obtain 
$$
(\partial h)_x+(\bar\partial h)_x=\de_x h=h_xA_x+A_x^*h_x
$$
for every $x\in X$. 
On the other hand $A\in\Omega^{(1,0)}(X,\Bc(\Ec))$ by Proposition~\ref{complex}\eqref{complex_item2}, 
hence the above equation is equivalent to 
\begin{equation}\label{chern_proof_eq1}
(\partial h)_x=h_xA_x\text{ and }(\bar\partial h)_x=A_x^*h_x.
\end{equation}
The first of these equations is clearly equivalent to~\eqref{chern_eq1}. 

To prove the existence, just note that the connection form defined by \eqref{chern_eq1} 
belongs to $\Omega^{(1,0)}(X,\Bc(\Ec))$, hence it is compatible with the complex structures 
by Proposition~\ref{complex}. 
On the other hand, if we define $A$ by the formula in the statement 
and we use the above formulas we see that Assertion~\eqref{complex_item2} of Proposition~\ref{complex} 
holds true, and then by that proposition we see that the covariant derivative corresponding to the connection 
form $A$ is compatible with 
the Hermitian structure as well. 
\end{proof}

Before we go further, let us establish the infinite-dimensional version of \cite[Ch. III, Prop. 2.2]{We08}. 

\begin{proposition}\label{curvature}
Assume the setting of Lemma~\ref{chern} and let $\Theta\in\Omega^2(X,\Bc(\Ec))$ be 
the curvature form corresponding to~$A$. 
Then the following assertions hold: 
\begin{enumerate}
 \item\label{curvature_item1} We have $A\in\Omega^{(1,0)}(X,\Bc(\Ec))$ and $\partial A=-A\wedge A$. 
\item\label{curvature_item2} We have $\Theta=\bar\partial A\in\Omega^{(1,1)}(X,\Bc(\Ec))$. 
\end{enumerate}
\end{proposition}

\begin{proof}
We have $h\cdot h^{-1}=\1$, hence 
$\partial h\cdot h^{-1}+h\partial(h^{-1})=0$. 
Therefore 
$$
\partial(h^{-1})=-h^{-1}\cdot\partial h\cdot h^{-1}
$$ 
and now by \eqref{chern_eq1} we get 
$$\partial A=\partial (h^{-1}\cdot \partial h)=\partial(h^{-1})\wedge \partial h+h^{-1}\cdot\partial^2 h
=-h^{-1}\cdot\partial h\cdot h^{-1}\wedge\partial h =-A\wedge A$$
We have already seen in Proposition~\ref{complex} that $A\in\Omega^{(1,0)}(X,\Bc(\Ec))$. 
Hence  
$$
\Theta=\de A+A\wedge A=\bar\partial A+\partial A+A\wedge A=\bar\partial A\in\Omega^{(1,1)}(X,\Bc(\Ec))
$$
and this concludes the proof. 
\end{proof}

\begin{remark}
\normalfont 
As above, let $X$ be any open subset of the complex Banach space $\Xc$ and $\Ec$ be any complex Hilbert space. 
In Proposition~\ref{curvature}\eqref{curvature_item2}, 
recall that the curvature property $\Theta\in\Omega^{(1,1)}(X,\Bc(\Ec))$ means that for every $x\in X$ 
the map $\Theta_z\colon \Xc\times\Xc\to\Bc(\Ec)$ is sesquilinear 
(more precisely, is $\CC$-linear in the first variable and conjugate linear in the second). 
\end{remark}

\begin{theorem}\label{main_lin}
Each Hermitian holomorphic vector bundle has a unique Chern covariant derivative. 
\end{theorem}

\begin{proof}
The existence in the case of the trivial bundles, 
as well as the uniqueness in the general case 
follow by Lemma~\ref{chern}, by using a family of local holomorphic trivializations. 
See for instance \cite{We08} or \cite{De12} for the proof of the existence 
in the classical situation of finite-dimensional vector bundles. 
The full details of the proof in the general case will be included in a forthcoming paper. 
\end{proof}

\begin{example}
\normalfont
In Example~\ref{ex3},  
the bundle $\Pi_{\Hc_A,\Hc_B}\colon\Tc_{\Hc_B}(\Hc_A)\to\Gr_{\Hc_B}(\Hc_A)$ 
is a Hermitian holomorphic vector bundle, 
hence it carries a unique Chern covariant derivative $\nabla_{\Hc_A,\Hc_B}$ by Theorem~\ref{main_lin}. 
See for instance \cite[Ch. III, Ex. 2.4]{We08} for more details 
on that covariant derivative in the case when $\dim\Hc_A<\infty$.
\end{example}

\begin{problem}\label{forth}
\normalfont
It would be interesting to establish a version of the Koszul-Malgrange integrability theorem of \cite{KM58}   
(see also \cite[Th. 5.1]{AHS78}) for Banach vector bundles (with infinite-dimensional base). 
Some results in this direction were recently obtained in \cite{DP12} and \cite{Ne13}. 
\end{problem}

\subsection*{Some computations of Chern covariant derivatives}
In the following proposition we denote by $\Sg_2(\Hc_1,\Hc_2)$ the complex Hilbert space 
consisting of the Hilbert-Schmidt operators from any complex Hilbert space $\Hc_1$ 
into another complex Hilbert space $\Hc_2$, with the usual scalar product on $\Sg_2(\Hc_1,\Hc_2)$ 
defined in terms of the operator trace. 
If $\Pi_j\colon D_j\to Z$ is any Hermitian vector bundle for $j=1,2$, 
then $\Sg_2(\Pi_1,\Pi_2)\colon D\to Z$ denotes the Hermitian vector bundle 
whose fiber over any $z\in Z$ is the space of Hilbert-Schmidt operators 
$\Sg_2(\Pi_1^{-1}(z),\Pi_2^{-1}(z))$. 
If $\Ec_j$ is the typical fiber of $\Pi_j$ for $j=1,2$ and 
$\Phi_j\colon V\to U(\Ec_j)$ gives a local change of coordinates in $\Pi_j$ over some open set $V\subseteq Z$, 
then $\Phi\colon V\to U(\Sg_2(\Ec_1,\Ec_2))$, $\Phi(z)T=\Phi_2(z)T\Phi_1(z)^{-1}$ 
for $z\in V$ and $T\in\Sg_2(\Ec_1,\Ec_2)$ 
gives a local change of coordinates in $\Sg_2(\Pi_1,\Pi_2)$. 

\begin{proposition}\label{HS}
Let $\Pi_j\colon D_j\to Z$ be any Hermitian holomorphic vector bundle 
with the Chern covariant derivative $\nabla_j$ for $j=1,2$. 
Then $\Sg_2(\Pi_1,\Pi_2)$ is a Hermitian holomorphic vector bundle 
with the Chern covariant derivative satisfying 
$\nabla\Gamma=\nabla_2\Gamma-\Gamma\nabla_1$ for every $\Gamma\in\Omega^0(Z,\Sg_2(\Pi_1,\Pi_2))$. 
\end{proposition}

\begin{proof}
The conclusion has a local character hence we may assume for $j=1,2$ 
that $\Pi_j\colon Z\times\Ec_j\to Z$ is a trivial vector bundle, where $\Ec_j$ is some complex Hilbert space. 
Let $h_j\colon Z\to\GL^{+}(\Ec_j)$ be the Hermitian structure of $\Pi_j$. 
Then it is easily checked that 
the Hermitian structure of $\Pi:=\Sg_2(\Pi_1,\Pi_2)\colon Z\times\Sg_2(\Ec_1,\Ec_2)\to Z$ 
is given by 
$$
H\colon Z\to\GL^{+}(\Sg_2(\Ec_1,\Ec_2)),\quad H(z)S=h_2(z)Sh_1(z)^{-1} 
\text{ if }S\in\Sg_2(\Ec_1,\Ec_2)\text{ and }z\in Z. 
$$

It follows that 
$$
H'(z)(\cdot)S
=h_2'(z)(\cdot)Sh_1(z)^{-1}
-h_2(z)Sh_1(z)^{-1}h_1'(z)(\cdot)h_1(z)^{-1}
\in\Bc_{\RR}(T_zZ,\Sg_2(\Ec_1,\Ec_2))
$$
hence 
$$
H(z)^{-1}H'(z)S=h_2(z)^{-1}h_2'(z)S-Sh_1(z)^{-1}h_1'(z).
$$
Therefore, if we denote by $A$, $A_1$, and $A_2$ the linear connection forms 
of $\nabla$, $\nabla_1$, and $\nabla_2$, respectively, 
then by using Lemma~\ref{chern} we obtain 
$$A(z)S=A_2(z)S-SA_1(z)\in\Bc(T_zZ,\Sg_2(\Ec_1,\Ec_2)) 
\text{ if }S\in\Sg_2(\Ec_1,\Ec_2)\text{ and }z\in Z. $$
Now the assertion follows easily since $\nabla=\de+A$ and $\nabla_j=\de+A_j$ for $j=1,2$. 
\end{proof}

We now compute the Chern covariant derivatives of holomorphic subbundles of Hermitian holomorphic vector bundles. 

\begin{proposition}\label{sub}
Let $\Pi\colon D\to Z$ be any Hermitian holomorphic vector bundle 
with a holomorphic vector subbundle $\Pi_1\colon D_1\to Z$ 
and its fiberwise orthogonal complement $\Pi_2\colon D_2\to Z$. 
For $j=1,2$ we regard $\Pi_j$ as a Hermitian holomorphic bundle with respect to the Hermitian structure 
induced from $\Pi$. 
Denote by $\nabla$, $\nabla_1$, and $\nabla_2$ the Chern covariant derivatives  
of $\Pi$, $\Pi_1$, and $\Pi_2$, respectively. 
Also let $\Theta$, $\Theta_1$, and $\Theta_2$ be the corresponding curvatures. 
Then with respect to the fiberwise orthogonal direct sum decomposition  
$D=D_1\oplus D_2$ we have 
$$
\nabla=\begin{pmatrix}
          \nabla_1 & -\beta^* \\
          \beta    & \nabla_2
         \end{pmatrix}
 $$
and 
$$
\Theta=\begin{pmatrix}
   \Theta_1-\beta^*\wedge\beta & \ast \\
                          \ast & \Theta_2-\beta\wedge\beta^* 
  \end{pmatrix}
$$
for a suitable form $\beta\in\Omega^{(1,0)}(Z,\Hom(D_1,D_2))$ 
where  $\beta^*\in\Omega^{(0,1)}(Z,\Hom(D_2,D_1))$ 
is its pointwise adjoint 1-form. 
\end{proposition}

\begin{proof}
With Proposition~\ref{HS} at hand, one can use the method of proof of \cite[Ch. V, Th. 14.3 and 14.5]{De12}; 
see also \cite[Ch. 0, Sect. 5, pages 73 and 78]{GH78}.
\end{proof}

\section{Positivity and global sections of holomorphic vector bundles}\label{Sect4}

For the sake of completeness, 
we include in this section a brief discussion on the properties of Griffiths positivity 
of holomorphic vector bundles. 
We refer to \cite{Gr69}, \cite{GH78}, \cite{La04}, and particularly to 
the elegant exposition in \cite{De12} for further details. 

\subsection*{Quotient tautological bundles}
We will give here some straightforward infinite-dimensional versions 
of certain constructions from~\cite[Ch. V, \S 16]{De12}. 
For any Hilbert space~$\Hc$ and any integer $k\ge 1$ we denote by $\Gr^{(k)}(\Hc)$ 
the set of all $k$-codimensional subspaces of~$\Hc$, 
which has the natural structure of a complex $U(\Hc)$-homogeneous Banach manifold. 
Recall that the tautological bundle over $\Gr^{(k)}(\Hc)$ is 
$$\Pi^{(k)}\colon\Tc^{(k)}(\Hc)\to\Gr^{(k)}(\Hc), \quad(\Sc,v)\mapsto\Sc,$$ 
where 
$$\Tc^{(k)}(\Hc)=\{(\Sc,v)\in\Gr^{(k)}(\Hc)\times\Hc\mid v\in\Sc\}\subseteq\Gr^{(k)}(\Hc)\times\Hc.$$
On the other hand, the \emph{quotient tautological bundle} over $\Gr^{(k)}(\Hc)$ is 
$$\Qc^{(k)}(\Hc)\to\Gr^{(k)}(\Hc), \quad(\Sc,v+\Sc)\mapsto\Sc,$$ 
where 
$$
\Qc^{(k)}(\Hc)=\{(\Sc,v+\Sc)\in\Gr^{(k)}(\Hc)\times(\Hc/\Sc)\mid v\in\Hc\}.
$$
Note that there is the short exact sequence of holomorphic vector bundles over $\Gr^{(k)}(\Hc)$
$$
\xymatrix{
0 \ar[r] & \Tc^{(k)}(\Hc) \ar[r] \ar[dr] & \Gr^{(k)}(\Hc)\times\Hc \ar[r] \ar[d] & \Qc^{(k)}(\Hc) \ar[r] \ar[dl] & 0 \\
 &  & \Gr^{(k)}(\Hc)  &  & 
}
$$
where the vertical arrow in the middle is the projection of the trivial bundle with the typical fiber~$\Hc$. 

\subsection*{Globally generated holomorphic vector bundles}
Unless otherwise specified, we let $\Pi\colon D\to Z$ be any holomorphic vector bundle whose fibers 
have the finite dimension~$k$ 
and whose base is a complex Banach manifold. 
Moreover, $\Oc(Z,D)$ stands for the space of global holomorphic sections endowed with the topology of uniform convergence  on compact sets, 
and we define the evaluation maps 
$$
(\forall z\in Z)\quad \ev_z\colon\Oc(Z,D)\to D_z,\ \ev_z(\sigma)=\sigma(z).
$$

\begin{definition}\label{glob}
\normalfont
The bundle $\Pi\colon D\to Z$ is \emph{globally generated} 
by the complex Hilbert space $\Hc$ if we have 
a continuous inclusion map $\Hc\hookrightarrow\Oc(Z,D)$ and 
for which for arbitrary $z\in Z$ we have $\ev_z(\Hc)=D_z$. 
\end{definition}

\begin{remark}
\normalfont
It follows by \cite[Ch. VII, Prop. 11.2]{De12} 
that if $Z$ is a finite-dimensional manifold, then the above notion of globally generated holomorphic vector bundle 
agrees with the one introduced in \cite[Ch. VII, Def. 11.1(a)]{De12}.  
\end{remark}

\begin{remark}\label{exact}
\normalfont
If the bundle $\Pi\colon D\to Z$ is globally generated 
by the complex Hilbert space $\Hc$, then we define 
$$\begin{aligned}
(\forall z\in Z)\quad N_z:= & \{(z,\sigma)\in Z\times \Hc\mid \sigma(z)=0\} \\
=&\{z\}\times \Ker(\ev_z\mid_{\Hc}) \\
\subseteq & Z\times\Hc
\end{aligned}$$
and $N:=\bigcup\limits_{z\in Z}N_z$. 

Now \emph{assume the fibers of $\Pi$ are finite-dimensional}. 
Then $N$ is the total space of a subbundle of the trivial Hermitian bundle $Z\times\Hc\to Z$.  
We have the fiberwise exact sequence of Hermitian bundles 

\begin{equation}\label{exact_eq1}
0\to N\hookrightarrow Z\times\Hc\mathop{\longrightarrow}\limits^{\ev} D\to0
\end{equation}
and the commutative diagram 
$$
\xymatrix{
D \ar[d]_{\Pi} \ar[r]^{\Psi_{\Hc}\quad} & \Qc^{(k)}(\Hc) \ar[d] \\
Z \ar[r]^{\psi_\Hc\quad} & \Gr^{(k)}(\Hc)}
$$
where for every $z\in Z$ we have 
$$(\forall\xi\in D_z)\quad \Psi_{\Hc}(\xi)=\{\sigma\in\Hc\mid\sigma(z)=\xi\}\in\Hc/\Ker(\ev_z\mid_{\Hc})$$
where we performed the identification 
$$\Hc/\Ker(\ev_z\mid_{\Hc})\simeq D_z,\quad 
\sigma+\Ker(\ev_z\mid_{\Hc})\mapsto \ev_z(\sigma).$$ 
\end{remark}

\subsection*{Positivity curvature condition} 
In order to introduce the positivity curvature condition on the covariant derivatives, 
we need the following remark, 
which is well known at least in the case of the scalar-valued bilinear maps.

\begin{remark}\label{three}
\normalfont
Let $\Vc$ be any complex Banach space and $\Ac$ be a complex associative Banach $*$-algebra, 
and denote $\Ac^{\sa}:=\{a\in\Ac\mid a^*=a\}$. 
We define the following spaces of bounded $\RR$-bilinear maps: 
\begin{itemize}
\item the space $\Herm(\Vc,\Ac)$ of all $\RR$-bilinear maps $\Psi\colon\Vc\times\Vc\to\Ac$ 
satisfying 
$$
(\forall v_1,v_2\in\Vc)\quad 
\Psi(v_1,v_2)^*=\Psi(v_2,v_1)=\ie\Psi(v_2,\ie v_1) 
$$
\item the space $\Symm(\Vc,\Ac)$ of all $\RR$-bilinear maps $\psi\colon\Vc\times\Vc\to\Ac$ 
satisfying 
$$
(\forall v_1,v_2\in\Vc)\quad 
\psi(v_1,v_2)=\psi(v_2,v_1)=\psi(\ie v_1,\ie v_2)
$$
\item the space $\Skew(\Vc,\Ac)$ of all $\RR$-bilinear maps $\omega\colon\Vc\times\Vc\to\Ac^{\sa}$ 
satisfying 
$$
(\forall v_1,v_2\in\Vc)\quad 
\omega(v_1,v_2)=-\omega(v_2,v_1)=\omega(\ie v_1,\ie v_2).
$$ 
\end{itemize}
If $\Psi\in\Herm(\Vc,\Ac)$, $\psi\in\Symm(\Vc,\Ac)$, and $\omega\in\Skew(\Vc,\Ac)$, 
then any of these three bilinear maps determines the other two maps in a unique manner such that 
the equation 
$$
(\forall v_1,v_2\in\Vc)\quad\Psi(v_1,v_2)=\psi(v_1,v_2)+\ie\omega(v_1,v_2)
$$
be satisfied, and \emph{canonical $\RR$-linear isomorphisms} are thus defined between the spaces 
$\Herm(\Vc,\Ac)$, $\Symm(\Vc,\Ac)$, and $\Skew(\Vc,\Ac)$, respectively. 

More precisely, the $\RR$-bilinear maps involved in the above equation are related by the formulas
$$
\begin{aligned} 
\omega(v_1,v_2)&=\psi(v_1,\ie v_2) \\
\psi(v_1,v_2)&=\frac{1}{2}(\Psi(v_1,v_2)+\Psi(v_2,v_1))\\
\omega(v_1,v_2)&=\frac{1}{2\ie}(\Psi(v_1,v_2)-\Psi(v_2,v_1))
\end{aligned}$$
for all $v_1,v_2\in\Vc$. 
\end{remark}

We now introduce the notion of Griffiths positivity of bundle-valued differential 2-forms, 
which goes back to \cite{Gr69}; see also \cite{GH78}, \cite{La04}, \cite{De12}. 
In the case of infinite-rank vector bundles, a version of this notion 
was also used in \cite{Ber09}. 

\begin{definition}\label{Griffiths}
\normalfont
Let $\Pi\colon D\to Z$ be any Hermitian holomorphic bundle. 
A bundle-valued differential form 
$\omega\in\Omega^2(Z,\End(\Pi))$ is \emph{Griffiths nonnegative} 
if for every $z\in Z$ the bounded $\RR$-bilinear map 
$\omega_z\colon T_zZ\times T_zZ\to \Bc(D_z)$ 
satisfies the conditions $\omega_z\in\Skew(T_zZ,\Bc(D_z))$ 
and $\Psi_z(x,x)\ge 0$ in $\Bc(D_z)$ for all $x\in T_zZ$, 
where $\Psi_z\in\Herm(T_zZ,\Bc(D_z))$ is the the sesquilinear map 
which canonically corresponds to $\omega_z$ via Remark~\ref{three}.  
If moreover for every $x\in T_zZ\setminus\{0\}$ 
we have $\Psi_z(x,x)\ne 0$, then we say that $\omega$ is \emph{Griffiths positive}. 
\end{definition}

In the case of the bundles with a finite-dimensional base, 
the following result can be found in \cite[Ch. 0, Sect. 5]{GH78} or \cite[Ch. VII, Cor. 11.5]{De12}.

\begin{theorem}\label{basic1}
Let $\Pi\colon D\to Z$ be any holomorphic vector bundle 
which is globally generated by the complex Hilbert space $\Hc\hookrightarrow\Oc(Z,D)$ 
and has finite-dimensional fibers. 
Then there exists a unique Hermitian structure on $\Pi$ 
such that for every $z\in\Hc$ the adjoint of the evaluation map $\ev_z\colon\Hc\to D_z$ 
is an isometry, 
and the curvature of the corresponding Chern covariant derivative is Griffiths nonnegative. 
\end{theorem}

\begin{proof}
One can use the method of proof from \cite[Ch. 0, Sect. 5]{GH78}, 
by relying on the above Proposition~\ref{HS} and Remark~\ref{exact}. 
\end{proof}

\section{Reproducing kernels and Griffiths positivity}\label{Sect5}

This section contains one of our main results, 
which is a necessary condition for existence of reproducing kernels on vector bundles 
(Theorem~\ref{final_th}). 
It relies on Griffith positivity properties of 
certain covariant derivatives associated with reproducing kernels on Hermitian holomorphic bundles 
which satisfy a certain admissibility condition. 
In order to introduce the latter notion, we need the following lemma.

\begin{lemma}\label{equivKinverse3}
In the setting of Definition~\ref{reprokernel}, consider the following assertions at 
an arbitrary point $s\in Z$:
\begin{enumerate}
\item\label{equivKinverse3_item1}
The operator $\widehat{K}\vert_{D_s}\colon D_s\to\Hc^K$ is injective and has closed range.
\item\label{equivKinverse3_item2} 
The operator $K(s,s)\in\Bc(D_s)$ is invertible.
\item\label{equivKinverse3_item2bis} 
The operator $K(s,s)\in\Bc(D_s)$ is surjective.
\item\label{equivKinverse3_item3} 
The evaluation map $\ev_s\colon\Hc^K\to D_s$ is surjective. 
\end{enumerate}
Then we have 
\eqref{equivKinverse3_item1}
$\iff$\eqref{equivKinverse3_item2}
$\iff$\eqref{equivKinverse3_item2bis}$\Longrightarrow$
\eqref{equivKinverse3_item3}, 
and all the above four assertions are equivalent if moreover $\dim D_s<\infty$. 
\end{lemma}

\begin{proof} 
The equivalence \eqref{equivKinverse3_item1}$\iff$\eqref{equivKinverse3_item2} 
was established in \cite[Lemma 3.4]{BG13}. 
Moreover we have \eqref{equivKinverse3_item2}$\iff$\eqref{equivKinverse3_item2bis} 
since $K(s,s)$ is always a bounded (nonnegative) self-adjoint operator on the complex Hilbert space $D_s$, 
as a consequence of \eqref{reprokernel_eq1} in Definition~\ref{reprokernel} for $n=1$, 
hence $\Ker(K(s,s))=(\Ran(K(s,s)))^\perp$. 

Next, for every $\xi\in D_s$ we have $\widehat{K}(\xi)=K_\xi=K(\cdot,s)\xi$ 
hence $(\ev_s\circ\widehat{K}\vert_{D_s})(\xi)=K(s,s)\xi$. 
This shows that \eqref{equivKinverse3_item2bis}$\Longrightarrow$\eqref{equivKinverse3_item3}.

Now note that for all $t\in Z$, $\eta\in D_t$, and $\xi\in D_s$ we have 
$$
\begin{aligned}
((\widehat{K}\vert_{D_s})(\xi)\mid K_\eta)_{\Hc^K}
& =(K_\xi\mid K_\eta)_{\Hc^K}
=(K(t,s)\xi\mid\eta)_{D_t}
=(\xi\mid K(s,t)\eta)_{D_s} \\
& =(\xi\mid\ev_s(K_\eta))_{D_s}  
  \end{aligned}
$$
hence the operators $\ev_s\colon\Hc^K\to D_s$ and $\widehat{K}\vert_{D_s}\colon D_s\to\Hc^K$ 
are adjoint to each other. 
This implies $\Ker(\widehat{K}\vert_{D_s})=(\Ran\ev_s)^\perp$. 
Therefore, if \eqref{equivKinverse3_item3} holds true, 
then $\widehat{K}\vert_{D_s}\colon D_s\to\Hc^K$ is injective, 
and if moreover $\dim D_s<\infty$, then the range of $\widehat{K}\vert_{D_s}$ 
is in turn finite-dimensional hence is a closed subspace of $\Hc^K$, 
and thus \eqref{equivKinverse3_item1} also holds true. 
This concludes the proof. 
\end{proof}

The following is a special case of \cite[Def. 3.5]{BG13}. 

\begin{definition}\label{adm_def_finite}
\normalfont 
Assume $\Pi\colon D\to Z$ is a Hermitian bundle whose fibers are finite dimensional 
(for instance, $\Pi$ is a line bundle). 
A reproducing kernel $K$ on $\Pi$ 
 is called \emph{admissible} 
if it has the following properties: 
\begin{itemize}
\item[(a)] The kernel $K$ is smooth as a section of the bundle $\rm{Hom}(p_2^*\Pi,p_1^*\Pi)$. 
\item[(b)] For every $s\in Z$ the operator $K(s,s)\in\Bc(D_s)$ is invertible.  
\end{itemize} 
\end{definition}

\begin{remark}
\normalfont
In the setting of Definition~\ref{adm_def_finite}, the admissible reproducing kernel $K$ 
has the additional property that 
the mapping $\zeta_K\colon Z\to\Gr(\Hc^K)$ is smooth. 
See \cite[Ex. 3.6]{BG13} for details. 
\end{remark}

If $\Pi\colon D\to Z$ is any Hermitian holomorphic  bundle with the space of holomorphic sections 
denoted by $\Oc(Z,D)$  
and $K$ is any reproducing kernel on $\Pi$, then we say that $K$ is \emph{holomorphic} if 
for every $\xi\in D$ we have $K_\xi\in\Oc(Z,D).$

\begin{theorem}\label{final_th}
Let $\Pi\colon D\to Z$  be any Hermitian holomorphic vector bundle with finite-dimensional fibers 
and with its Hermitian structure denoted by $\{(\cdot\mid\cdot)_z\}_{z\in Z}$. 
If $K$ is a holomorphic admissible reproducing kernel on $\Pi$,  
then $\{(K(z,z)\,\cdot\mid\cdot)_z\}_{z\in Z}$
is a new Hermitian structure on $\Pi$,   
for which the curvature of its Chern covariant derivative is Griffiths positive. 
\end{theorem}

\begin{proof}
Let $\Hc^K$ be the reproducing kernel Hilbert space associated to~$K$. 
By hypothesis $K$ is a holomorphic reproducing kernel on $\Pi$,   
hence we have a continuous inclusion map $\Hc^K\hookrightarrow\Oc(Z,D)$. 
Since $K$ is admissible, Lemma~\ref{equivKinverse3} implies 
that the evaluation map $\ev_s\colon\Hc^K\to D_s$ is surjective 
for arbitrary $s\in Z$. 
Thus the bundle $\Pi$ is globally generated in the sense of Definition~\ref{glob}, 
and then the conclusion follows by Theorem~\ref{basic1}. 
\end{proof}

\begin{example}
\normalfont
Consider the special case of Example~\ref{ex1} 
with $\Ec=\CC$, 
$$
Z=\DD=\{z\in\CC\mid\vert z\vert<1\}
$$ 
and for every $\nu\ge1$ define 
$$
K_{\DD}^{(\nu)}\colon \DD\times \DD\to\CC,\quad K_{\DD}^{(\nu)}(z_1,z_2)=\frac{1}{(1-z_1\bar{z}_2)^\nu}
$$
which is the reproducing kernel of the Bergman space on the unit disc if $\nu>1$  
and of the Hardy space if $\nu=1$; 
see also \cite[subsect. 5.1, (b.1)]{BG13}.

The new Hermitian structure on the trivial bundle $\Pi\colon D= Z\times\CC\to Z$ 
(referred to in Theorem~\ref{final_th}) 
is in this case given by 
$$
h^{(\nu)}\colon \DD\to\GL^+(\Ec)\simeq(0,\infty),\quad 
h^{(\nu)}(z)=\frac{1}{(1-\vert z\vert^2)^\nu}=\frac{1}{(1-z\bar{z})^\nu}.
$$
We have 
$$
\frac{1}{h^{(\nu)}(z)}\cdot\partial h^{(\nu)}(z)
=(1-z\bar{z})^\nu\cdot \frac{\nu\bar{z}}{(1-z\bar{z})^{\nu+1}}
=\frac{\nu\bar{z}}{1-z\bar{z}}
$$
hence by using Lemma~\ref{chern} we obtain the following expression for 
the linear connection form of the Chern covariant derivative 
corresponding to the Hermitian structure determined by $h$: 
$$
A^{(\nu)}\colon\DD\to\Bc(\CC),\quad 
A^{(\nu)}_z=\frac{\nu\bar{z}}{1-z\bar{z}}\de z. 
$$
Since 
$$
\frac{\partial}{\partial\bar{z}}\Bigl(\frac{\bar{z}}{1-z\bar{z}}\Bigr)
=\frac{(1-z\bar{z})-\bar{z}\cdot(-z)}{(1-z\bar{z})^2}=\frac{1}{(1-z\bar{z})^2} 
$$
it then follows by Proposition~\ref{curvature}\eqref{curvature_item2} 
that the curvature of the aforementioned Chern covariant derivative is the 2-form
$$
\Theta^{(\nu)}\colon\DD\to\Bc_{\RR}(\CC\wedge\CC,\CC),\quad 
\Theta^{(\nu)}_z=\frac{\nu}{(1-\vert z\vert^2)^2}\,\de \bar{z}\wedge\de z
$$
which clearly is Griffiths positive. 
It also follows by the above formula that the curvature $\Theta^{(\nu)}$ depends linearly 
on $\nu$, and for all $\nu\ge 1$ we have $\Theta^{(\nu)}=\nu\Theta^{(1)}$. 
\end{example}

\begin{example}
\normalfont 
Recall that if $\Hc$ is any complex Hilbert space,  
then the mapping $h\mapsto (\cdot\mid h)_{\Hc}$ 
is an antilinear isometric isomorphism from $\Hc$ onto its topological dual $\Hc^*$. 
For this reason $\Hc^*$ will be alternatively described as the complex Hilbert space whose 
underlying structure of real Hilbert space is the one of $\Hc$, 
while the complex structure is the opposite to the complex structure of $\Hc$; 
that is, we may assume $\Hc^*=\Hc$ as real vector spaces, with the complex scalar products 
related by $(h_1\mid h_2)_{\Hc^*}=(h_2\mid h_1)_{\Hc}$ for all $h_1,h_2\in\Hc=\Hc^*$.  
So $\Hc$ and $\Hc^*$ have the same closed complex subspaces 
and the identity map is an antiholomorphic diffeomorphism between  
their Grassmann manifolds $\Gr(\Hc)$ and $\Gr(\Hc^*)$. 

With this convention, if $\Pi\colon D\to Z$ is a holomorphic Hermitian bundle 
and $D^*:=\bigsqcup\limits_{s\in Z} D_s^*=\bigsqcup\limits_{s\in Z} D_s=D$ as real manifolds, 
then the dual bundle $\Pi^*\colon D^*\to Z$ is again a holomorphic Hermitian bundle, 
whose complex structure is fiberwise the opposite to the complex structure of $D$, 
while both mappings $\Pi$ and $\Pi^*$ are holomorphic 
(i.e., they are smooth and the differentials are $\CC$-linear) 
onto the same complex manifold $Z$. 
In particular, if $K$ is a reproducing kernel on $\Pi$, then it is also a reproducing kernel 
on $\Pi^*$, to be denoted by $K^*$, 
and it follows by \eqref{reprokernel_eq2} that the corresponding reproducing kernel Hilbert spaces 
are related by $\Hc^{K^*}=(\Hc^K)^*$. 
In addition, one can also check that $K$ is admissible if and only if $K^*$ is. 
In this case, if $\Theta\in\Omega^2(Z,\Pi)$ and $\Theta^*\in\Omega^2(Z,\Pi^*)$ are the curvatures of the Chern connections 
associated to $K$ and $K^*$ as in Theorem~\ref{final_th}, respectively, 
then by using Proposition~\ref{curvature}\eqref{curvature_item2} along with equations \eqref{chern_eq1}--\eqref{chern_proof_eq1} one can show that $\Theta^*=-\Theta$.
\end{example}

By using a suitable method of localization of reproducing kernel Hilbert spaces on vector bundles, 
one can
obtain infinite-dimensional versions of the properties 
of Bergman kernels established in \cite{MP97}. 
Put $\delta_K:=(\zeta_K\circ\Pi,\widehat K)$, where $\zeta_K$ and 
$\widehat K$ are as in Definition~\ref{reprokernel}.

\begin{theorem}\label{holom}
Let $\Pi\colon D\to Z$ be a holomorphic Hermitian bundle
with finite-dimensional fibers 
and $K$ be 
a holomorphic admissible reproducing kernel on $\Pi$. 
Then the following assertions hold:
\begin{enumerate}
\item The mapping $\widehat{K}\colon D^*\to(\Hc^K)^*$ is holomorphic.
\item The pair $\Delta_K=(\delta_K,\zeta_K)$ 
is a holomorphic morphism of vector bundles from $\Pi^*\colon D^*\to Z$ 
to  $\Pi_{(\Hc^K)^*}\colon\Tc((\Hc^K)^*)\to\Gr((\Hc^K)^*)$.
\end{enumerate}
\end{theorem}

\begin{proof}
By the note prior to the theorem, it suffices to prove the assertions under the assumption that 
$\Pi\colon D\to Z$ is a trivial vector bundle, say $D=Z\times \Vc$, 
where the complex Hilbert space $\Vc$ is the typical fiber. 
Recall from \cite[Lemma 3.3]{BG13} that $\widehat{K}\colon D\to\Hc^K$ is smooth. 
On the other hand, since $K$ is an admissible holomorphic reproducing kernel, 
it is given by an operator-valued reproducing kernel 
$\kappa\colon Z\times Z\to\Bc(\Vc)$ (see for instance \cite[subsect. 5.1]{BG13}) such that 
 $\kappa(\cdot,t)\in\Oc(Z,\Vc)$ and $\kappa(t,t)\in\Bc(\Vc)$ is invertible 
for every $t\in Z$. 
Then for every $\eta=(t,v)\in Z\times \Vc=D$ we have 
$$\widehat{K}(\eta)=(\cdot,\kappa(\cdot,t)v)\in\Oc(Z,D)$$ 
which implies at once that the differential of 
$\widehat{K}\colon D^*=Z\times\Vc^*\to(\Hc^K)^*$ at every point is $\CC$-linear, 
hence the mapping $\widehat{K}\colon D^*\to(\Hc^K)^*$ is holomorphic. 
Since $\kappa(t,\cdot)^*=\kappa(\cdot,t)\in\Oc(Z,\Vc)$ 
and $\zeta_K(t)=\widehat{K}(\{t\}\times\Vc)$, it also follows by the above formula 
for $\widehat{K}(\eta)$ that $\zeta_K\colon Z\to\Gr((\Hc^K)^*)$ is holomorphic, 
since it is smooth by the assumption that $K$ is admissible, and its tangent map at any point is $\CC$-linear.
\end{proof}

In connection with Theorem~\ref{holom}, we note that a certain 
procedure to associate linear connections $\Phi_K$ to reproducing kernels $K$ on infinite-dimensional vector bundles $\Pi$ was established in \cite{BG13}. 
That method relies on canonical pullback operations by starting from tautological bundles on Grassmann manifolds.
Then one can also prove that 
the linear connection $\Phi_{K^*}$ associated with $K^*$ is compatible 
both with the complex structure of the dual bundle $\Pi^*\colon D^*\to Z$ 
and with the Hermitian structure 
$\{(K^*(s,s)\cdot\mid\cdot)^*_s\}_{s\in Z}$ where $(K^*(s,s)\cdot\mid\cdot)_s^*:=\overline{(K(s,s)\cdot\mid\cdot)_s}$  
for all $s\in Z$.   
That is, on dual vector bundles of vector bundles with reproducing kernel (and finite-dimensional fibers), the covariant derivatives associated with the linear connections defined in \cite{BG13} are also examples of the Griffiths-positive Chern derivatives of Theorem~\ref{final_th} above. 
As the details of these results are beyond the scope of the present work, 
we defer them to a forthcoming paper.

\begin{remark}
\normalfont 
It is well known that there also exist holomorphic vector bundles with finite-dimensional fibers 
which 
do not admit any nontrivial global holomorphic cross-section, 
so they do not carry any reproducing kernel satisfying the hypothesis of Theorem~\ref{final_th}. 
An example in this sense (with 1-dimensional fibers) 
is provided by the tautological vector bundle over the projective space, 
in the above notation 
$$
\Pi^{(n)}\colon\Tc^{(n)}(\CC^{n+1})\to\Gr^{(n)}(\CC^{n+1})
$$
see for instance \cite[Ch. V, Cor. 15.6]{De12}, where this line bundle was denoted by $\Oc(-1)$. 
\end{remark}

\begin{remark}
\normalfont
In connection with Theorem~\ref{final_th}, 
we recall that there exist several open problems on sufficient conditions for Griffiths positivity. 
There is for instance Griffiths'problem (\cite{Gr69}) 
which asks whether or not every ample holomorphic vector bundle with compact base is Griffiths positive. 

It is also unknown whether any holomorphic vector bundle is Griffiths positive 
if there exists some integer $k_0\ge 1$ such that the symmetric $k$th tensor powers of that bundle 
are globally generated for all $k\ge k_0$. 
See however \cite[Ch. VII, Cor. 11.6]{De12} for an affirmative answer to that problem in the case of the line bundles. 
From the perspective of the above Theorem~\ref{final_th}, 
the problem would be to construct reproducing kernels on some Hermitian vector bundle 
by using some reproducing kernels on symmetric tensor powers of that bundle. 
\end{remark}

\appendix
\section{Complements on vector-valued differential forms}\label{SectA}

\begin{definition}\label{def1}
\normalfont
Assume $X$ is any open subset of some real Banach space $\Xc$ and $\Vc$ is another real Banach space. 
The space of \emph{$\Vc$-valued differential forms of degree $p\ge0$} on $X$ is defined by 
$$\Omega^p(X,\Vc)=
\begin{cases}
\Ci(X,\Vc)&\text{ if }p=0 \\
\Ci(X,\Bc(\wedge^p\Xc,\Vc))&\text{ if }p\ge1,
\end{cases} $$
where $\Bc(\wedge^p\Xc,\Vc)$ is the space of bounded $p$-linear skew-symmetric maps 
$\Xc\times\cdots\times\Xc\to\Vc$. 
For every $\sigma\in\Omega^p(X,\Vc)$ and $x\in X$ we denote $\sigma_x:=\sigma(x)$. 

The \emph{exterior derivative} $d\colon\Omega^p(X,\Vc)\to\Omega^{p+1}(X,\Vc)$ 
is defined for $\sigma\in\Omega^p(X,\Vc)$ as follows: 
\begin{enumerate} 
\item If $p=0$, then for every $x\in X$ we set $(\de\sigma)_x=\sigma'_x\in\Bc(\Xc,\Vc)$. 
\item If $p\ge1$, then for every $x\in X$ 
we define $(\de\sigma)_x=\de_x\sigma\in\Bc(\wedge^p\Xc,\Vc)$ as 
a bounded skew-symmetric $(p+1)$-linear mapping $\Xc\times\cdots\times\Xc\to\Vc$ 
by the formula 
$$(\de\sigma)_x(x_1,\dots,x_{p+1})
=\sum_{j=1}^{p+1}(-1)^{j-1}(\sigma'_x(x_j))(x_1,\dots,x_{j-1},x_{j+1},\dots,x_{p+1}) $$
for every $x_1,\dots,x_{p+1}\in\Xc$. 
Note that $\sigma\colon X\to\Bc(\wedge^p\Xc,\Vc)$ is a smooth mapping, hence  $\sigma'_x\in\Bc(\Xc,\Bc(\wedge^p\Xc,\Vc))$. 
\end{enumerate}
We have $d^2=0$ 
(see for instance \cite{Lg01}) 
as an operator from $\Omega^p(X,\Vc)$ into $\Omega^{p+2}(X,\Vc)$ for every $p\ge 0$. 
\end{definition}

\begin{definition}\label{def2}
\normalfont
Now assume the following setting: 
\begin{itemize}
\item $X$ open set in the real Banach space $\Xc$; 
\item $\Vc_1$, $\Vc_2$, and $\Vc$ are real Banach spaces endowed with 
a continuous bilinear mapping $\Vc_1\times\Vc_2\to\Vc$ denoted simply by 
$(v_1,v_2)\mapsto v_1\cdot v_2$. 
\end{itemize}
Then for $p_1,p_2\ge0$ we define the \emph{exterior product} 
(cf.\ \cite[Subsect. I.4]{Ne06})
$$
\wedge\colon \Omega^{p_1}(X,\Vc_1)\times \Omega^{p_2}(X,\Vc_2)\to\Omega^{p_1+p_2}(X,\Vc)
$$
in the following way. 
Let $\sigma_j\in\Omega^{p_j}(X,\Vc_j)$, $j=1,2$, 
and $x\in X$.  
\begin{enumerate} 
\item If $p_1=0$, then  we set
$$(\sigma_1\wedge\sigma_2)_x(x_1,\dots,x_{p_2})
=\underbrace{(\sigma_1)_x}_{\hskip15pt\in\Vc_1}\cdot
\underbrace{(\sigma_2)_x(x_1,\dots,x_{p_2})}_{\hskip15pt\in\Vc_2}\in\Vc$$
for $x_1,\dots,x_{p_2}\in\Xc$. 
We proceed in a similar manner if $p_2=0$. 
\item If $p_1,p_2\ge1$, then 
$$\begin{aligned}
(\sigma_1\wedge\sigma_2)_x(x_1,\dots,x_{p_1+p_2})
=\frac{1}{p_1!p_2!}\sum_{\tau}
\epsilon(\tau) &
\underbrace{(\sigma_1)_x(x_{\tau(1)},\dots,x_{\tau(p_1)})}_{\hskip15pt\in\Vc_1} \\
& \cdot 
\underbrace{(\sigma_2)_x(x_{\tau(p_1+1)},\dots,x_{\tau(p_1+p_2)})}_{\hskip15pt\in\Vc_2} 
\end{aligned}$$
where the sum is taken for every permutation $\tau$ of the set $\{1,\dots,p_1+p_2\}$ and $\epsilon(\tau)\in\{\pm1\}$ denotes the signature of $\tau$ 
(compare \cite[Ch. V, \S 3]{Lg01}). 
\end{enumerate}
Just as in the scalar-valued case, one can check the formula 
$$\de(\sigma_1\wedge\sigma_2)=\de\sigma_1\wedge\de\sigma_2+(-1)^{p_1}\sigma_1\wedge\de\sigma_2 $$
for $\sigma_j\in\Omega^{p_j}(X,\Vc_j)$ and $j=1,2$. 
\end{definition}

\begin{remark}\label{compl}
\normalfont
Let $\Yc$ and $\Vc$ be complex Banach spaces 
and denote by $\Bc_{\RR}(\Yc,\Vc)$ 
the Banach space of bounded $\RR$-linear operators from $\Yc$ into $\Vc$. 
Also denote by 
$$\Bc_{\RR}^{(0,1)}(\Yc,\Vc)=\{T\in\Bc_{\RR}(\Yc,\Vc)\mid 
(\forall x\in\Yc)\quad T(\ie x)=-\ie Tx\}$$
the space of bounded conjugate-linear operators from 
$\Yc$ into $\Vc$, 
and use for the moment the notation $\Bc_{\RR}^{(1,0)}(\Yc,\Vc):=\Bc(\Yc,\Vc)$ 
for the space of $\CC$-linear operators. 

Then we note the following facts:

\begin{enumerate}
\item\label{compl_item1} 
The mapping 
\begin{equation}\label{compl_eq1}
\Bc_{\RR}^{(1,0)}(\Yc,\Vc)\times \Bc_{\RR}^{(0,1)}(\Yc,\Vc)\to \Bc_{\RR}(\Yc,\Vc), 
\quad (R,S)\mapsto R+S
\end{equation}
is a linear topological isomorphism. 
Indeed, it is clear that this mapping is linear and continuous, 
hence it suffices to prove that it is bijective. 
In fact it is easily checked that for every $T\in\Bc_{\RR}(\Yc,\Vc)$ 
there exist uniquely determined operators $T^{(1,0)}\in\Bc_{\RR}^{(1,0)}(\Yc,\Vc)$  and $T^{(0,1)}\in\Bc_{\RR}^{(0,1)}(\Yc,\Vc)$ such that $T=T^{(1,0)}+T^{(0,1)}$, namely 
$$T^{(1,0)}x=\frac{1}{2}(Tx-\ie T(\ie x))\text{ and }
T^{(0,1)}x=\frac{1}{2}(Tx+\ie T(\ie x)) $$
for every $x\in\Yc$. 
\item\label{compl_item2} 
Each of the spaces 
$\Bc_{\RR}^{(1,0)}(\Yc,\Vc)$ and $\Bc_{\RR}^{(0,1)}(\Yc,\Vc)$ 
has a natural structure of \emph{complex} Banach space, 
defined by multiplying the values of any operator by complex numbers 
(This is possible since $\Vc$ is a complex Banach space). 
\item\label{compl_item3} 
By using the above items \eqref{compl_item1} and \eqref{compl_item2}, 
one can obtain direct sum decompositions similar to \eqref{compl_eq1} 
for spaces of $\RR$-multilinear mappings in a higher number of variables. 
For instance, for bilinear mappings we have 
\allowdisplaybreaks
\begin{align}
\Bc_{\RR}(\Yc\hotimes_{\RR}\Yc,\Vc)
\simeq
&\Bc_{\RR}(\Yc,\Bc_{\RR}(\Yc,\Vc)) \nonumber\\
\simeq
&\Bc_{\RR}(\Yc,\Bc_{\RR}^{(1,0)}(\Yc,\Vc)\dotplus\Bc_{\RR}^{(0,1)}(\Yc,\Vc)) 
\nonumber\\
\simeq
&\Bc_{\RR}(\Yc,\Bc_{\RR}^{(1,0)}(\Yc,\Vc))
\dotplus\Bc(\Yc,\Bc_{\RR}^{(0,1)}(\Yc,\Vc))\nonumber\\
\simeq 
&\Bc_{\RR}^{(1,0)}(\Yc,\Bc_{\RR}^{(1,0)}(\Yc,\Vc))
\dotplus \Bc_{\RR}^{(1,0)}(\Yc,\Bc_{\RR}^{(0,1)}(\Yc,\Vc)) \nonumber\\
&\dotplus \Bc_{\RR}^{(0,1)}(\Yc,\Bc_{\RR}^{(1,0)}(\Yc,\Vc))
\dotplus \Bc_{\RR}^{(0,1)}(\Yc,\Bc_{\RR}^{(0,1)}(\Yc,\Vc))
\nonumber
\end{align}
\item\label{compl_item4} 
We now use the above remarks to obtain a direct sum decomposition 
for the space of skew-symmetric $\RR$-bilinear maps from $\Yc\times\Yc$ 
into $\Vc$. 
Let us consider the bounded $\RR$-linear operator  
$$
A\colon\Yc\hotimes_{\RR}\Yc\to\Yc\hotimes_{\RR}\Yc,
\quad A(y_1\otimes y_2)=\frac{1}{2}(y_1\otimes y_2-y_2\otimes y_1). 
$$
Then we have $A^2=A$, and we define $\Yc\wedge\Yc:=\Ran A$. 
We also define 
$$\Ac\colon \Bc_{\RR}(\Yc\hotimes_{\RR}\Yc,\Vc)\to\Bc_{\RR}(\Yc\hotimes_{\RR}\Yc,\Vc), 
\quad \Ac(\Phi):=\Phi\circ A$$
and then $\Ac^2=\Ac$ and it is easily seen that 
$$
\Ran\Ac=\{\Phi\in\Bc_{\RR}(\Yc\hotimes_{\RR}\Yc,\Vc)\mid 
(\forall y_1,y_2\in\Yc)\quad \Phi(y_1\otimes y_2)=-\Phi(y_2\otimes y_1) \}.
$$ 
In order to study the behavior of $\Ac$ with respect to the direct sum decomposition established in \eqref{compl_item3} above, 
we introduce the operator 
$$
\Qc\colon 
\Bc_{\RR}(\Yc\hotimes_{\RR}\Yc,\Vc)\to\Bc_{\RR}(\Yc\hotimes_{\RR}\Yc,\Vc), 
\quad (\Qc(\Phi))(y_1\otimes y_2)=\Phi((\ie y_1)\otimes(\ie y_2)).
$$
Then it is easily seen that  
$\Qc^2=\id$ and $\Qc\Ac=\Ac\Qc$, hence 
we have the direct sum decomposition 
\begin{equation}\label{compl_eq2}
\Bc_{\RR}(\Yc\hotimes_{\RR}\Yc,\Vc)=\Ker(\Qc-\id)\dotplus\Ker(\Qc+\id) 
\end{equation}
and both subspaces involved in this decomposition are invariant under $\Ac$. 
On the other hand, it is easily seen that 
$$\begin{aligned}
\Bc_{\RR}^{(1,0)}(\Yc,\Bc_{\RR}^{(1,0)}(\Yc,\Vc))
\dotplus \Bc_{\RR}^{(0,1)}(\Yc,\Bc_{\RR}^{(0,1)}(\Yc,\Vc)) 
&\subseteq\Ker(\Qc+\id), \\  
\Bc_{\RR}^{(1,0)}(\Yc,\Bc_{\RR}^{(0,1)}(\Yc,\Vc)) \dotplus \Bc_{\RR}^{(0,1)}(\Yc,\Bc_{\RR}^{(1,0)}(\Yc,\Vc))
&\subseteq\Ker(\Qc-\id). 
\end{aligned}$$
It then follows by \eqref{compl_eq2} and the decomposition established above in \eqref{compl_item3} that the above inclusions are actually equalities. 
In particular, the space 
$$\Bc_{\RR}^{(1,0)}(\Yc,\Bc_{\RR}^{(0,1)}(\Yc,\Vc)) \dotplus \Bc_{\RR}^{(0,1)}(\Yc,\Bc_{\RR}^{(1,0)}(\Yc,\Vc))$$ 
is invariant under $\Ac$, 
and we denote by $\Bc_{\RR}^{(1,1)}(\Yc\hotimes\Yc,\Vc)$ the image of the corresponding restriction of $\Ac$. 
On the other hand, it is easily checked that each of the spaces 
$\Bc_{\RR}^{(1,0)}(\Yc,\Bc_{\RR}^{(1,0)}(\Yc,\Vc))$ 
and 
$\Bc_{\RR}^{(0,1)}(\Yc,\Bc_{\RR}^{(0,1)}(\Yc,\Vc))$ 
is invariant under $\Ac$, and 
$\Bc_{\RR}^{(2,0)}(\Yc\hotimes\Yc,\Vc)$ 
and $\Bc_{\RR}^{(0,2)}(\Yc\hotimes\Yc,\Vc)$ will denote the images of the corresponding restrictions of 
$\Ac$, 
respectively.
We thus get a direct sum decomposition 
$$\Ran\Ac=\Bc_{\RR}^{(2,0)}(\Yc\hotimes\Yc,\Vc)\dotplus
\Bc_{\RR}^{(1,1)}(\Yc\hotimes\Yc,\Vc)\dotplus 
\Bc_{\RR}^{(0,2)}(\Yc\hotimes\Yc,\Vc).$$
Note the natural isomorphism 
$\Ran\Ac\simeq\Bc_{\RR}(\Yc\wedge\Yc,\Vc)$, $\Phi\mapsto\Phi\vert_{\Ran A}$, 
and thus the above decomposition gives rise to a direct sum decomposition 
$$\Bc_{\RR}(\Yc\wedge\Yc,\Vc)=\Bc_{\RR}^{(2,0)}(\Yc\wedge\Yc,\Vc)\dotplus
\Bc_{\RR}^{(1,1)}(\Yc\wedge\Yc,\Vc)\dotplus 
\Bc_{\RR}^{(0,2)}(\Yc\wedge\Yc,\Vc). $$
For every $\Phi\in\Bc_{\RR}(\Yc\wedge\Yc,\Vc)$ we denote by 
$\Phi=\Phi^{(2,0)}+\Phi^{(1,1)}+\Phi^{(0,2)}$ the corresponding decomposition. 
This is the bilinear version of the decomposition established above in~\eqref{compl_item1}. 
\end{enumerate}
\normalfont
\end{remark}

\begin{definition}\label{dbar}
\normalfont
Let $\Xc$ and $\Vc$ be complex Banach spaces. 
If $p\in\{1,2\}$,  
then by using Remark~\ref{compl} for the values of the differential forms in $\Omega^p(X,\Vc)$, we get 
the direct sum decompositions 
\begin{equation}\label{dbar_eq1}
\Omega^1(X,\Vc)=\Omega^{(1,0)}(X,\Vc)\dotplus\Omega^{(0,1)}(X,\Vc) 
\end{equation} 
and 
$$
\Omega^2(X,\Vc)=\Omega^{(2,0)}(X,\Vc)\dotplus\Omega^{(1,1)}(X,\Vc)\dotplus\Omega^{(0,2)}(X,\Vc). 
$$
By using these decompositions we can define the operators 

$$
\bar\partial\colon\Ci(X,\Vc)\to\Omega^{(0,1)}(X,\Vc)
\text{ and }\bar\partial\colon\Omega^{(r,1-r)}(X,\Vc)\to\Omega^{(r,2-r)}(X,\Vc)
$$ 
for $r\in\{0,1\}$ as the corresponding projections of the exterior derivatives 
$$
\de\colon\Ci(X,\Vc)\to\Omega^1(X,\Vc)
\text{ and }\de\colon\Omega^1(X,\Vc)\to\Omega^2(X,\Vc),$$ 
respectively.  
We also define $\partial:=\de-\bar\partial$ on any of the above spaces where $\bar\partial$ is defined. 
\end{definition}

\begin{remark}\label{partial}
\normalfont
In the setting of Definition~\ref{dbar}, we have $\bar\partial^2=0$, $\partial^2=0$, 
and 
$$
\de\colon\Omega^{(r,s)}(X,\Vc)\to\Omega^{(r+1,s)}(X,\Vc)\dotplus\Omega^{(r,s+1)}(X,\Vc),
$$
hence $\partial\colon\Omega^{(r,s)}(X,\Vc)\to\Omega^{(r+1,s)}(X,\Vc)$ for $r,s\in\{0,1\}$ with $r+s=1$. 
\end{remark}

\begin{remark}\label{hol}
\normalfont
In Definition~\ref{dbar}, a function $\sigma\in\Ci(X,\Vc)$ is holomorphic if and only if its differential is $\CC$-linear at every point of $X$, which is equivalent to the Cauchy-Riemann equation $\bar\partial\sigma=0$.
\end{remark}

We refer to \cite[Sect. 2]{Ll98} for a definition of the Dolbeault operator $\bar\partial$ in a more general setting.

\end{document}